\documentclass[11pt]{amsart}

\usepackage[english]{babel}
\usepackage[colorinlistoftodos,bordercolor=orange,backgroundcolor=orange!20,linecolor=orange,textsize=small]{todonotes}

\usepackage{filecontents}
\usepackage[useregional]{datetime2}
\usepackage{tikz}
\usetikzlibrary{matrix,shapes,arrows,positioning}
\definecolor{burntorange}{cmyk}{0,0.52,1,0}
\usepackage{booktabs}
\usepackage{float}
\usepackage{mwe}

\usepackage{fullpage}
\usepackage{caption}
\usepackage{subcaption}
\usepackage{amsmath}
\usepackage{amsthm}
\usepackage{hyperref}
\usepackage{cleveref}
\usepackage{centernot}
\usepackage{lineno,hyperref}
\modulolinenumbers[5]
\usepackage{amsmath,amssymb}
\usetikzlibrary{arrows.meta}
\usepackage{blkarray}
\usepackage{float}
\usepackage[font=footnotesize,labelfont=bf]{caption}
\usepackage{tikz}

\usetikzlibrary{matrix,shapes,arrows,positioning}
\definecolor{burntorange}{cmyk}{0,0.52,1,0}
\modulolinenumbers[5]
\usepackage{forest}
\usepackage{stackengine}
\stackMath
\usepackage{booktabs}
\usepackage{matlab-prettifier}

\newtheorem{theorem}{Theorem}[section]
\newtheorem{lemma}[theorem]{Lemma}

\newtheorem{corollary}[theorem]{Corollary}
\theoremstyle{definition}
\newtheorem{definition}[theorem]{Definition}

\theoremstyle{remark}
\newtheorem{remark}[theorem]{Remark}
\newtheorem{example}[theorem]{Example}

\usepackage[foot]{amsaddr}
\makeatletter
\renewcommand{\email}[2][]{%
  \ifx\emails\@empty\relax\else{\g@addto@macro\emails{,\space}}\fi%
  \@ifnotempty{#1}{\g@addto@macro\emails{\textrm{(#1)}\space}}%
  \g@addto@macro\emails{#2}%
}
\makeatother

\usepackage{graphicx}

\newcommand{\new}[1]{{\em #1}}

\DeclareMathAlphabet{\mathbbold}{U}{bbold}{m}{n}

\newcommand{\R}{\mathbb R}

\newcommand{\rmax}{\mathbb{R}_{\max}}

\usepackage[scr=boondox,scrscaled=1.05]{mathalfa}

\newcommand{\trop}[1][]{\ifthenelse{\equal{#1}{}}{ \mathbb{T} }{ \mathbb{T}(#1) }}

\usepackage{amssymb}

\newcommand{\N}{\mathbb{N}}
\newcommand{\C}{\mathbb{C}}

\newcommand{\spn}{\mathrm{span}}

\begin{document}

\title{Generalized Eigenvectors and Rayleigh bounds for tropical algebraic eigenvalues}
\author{Dariush Kiani$^{\, 1}$}
\author{Hanieh Tavakolipour$^{\, 2}$}
\address[$1,2$]{Amirkabir University of Technology, Department of Mathematics and Computer Science}
\email[$1$]{dkiani@aut.ac.ir}
\email[$2$]{h.tavakolipour@aut.ac.ir}
\thanks{The study was funded by Iran National Science Foundation (INSF) (Grant No. 99023636).}
\date{\today}
\maketitle

\begin{abstract}
In this paper, we review the eigenpair problem in the context of tropical algebra. An important fact that has been largely overlooked in spectral theory of tropical algebra is that the tropical algebraic eigenvalues, which are obtained from the characteristic polynomial, may not correspond to any tropical eigenvector satisfying the standard eigenvalue-eigenvector equation. To resolve this, we use the tropical numerical range and define a generalized tropical eigenvalue–eigenvector relation. We define any non-zero vector satisfying this equation as a generalized tropical eigenvector. We show that a generalized tropical eigenvector always exists for any given tropical algebraic eigenvalue. We propose a computationally inexpensive method for the construction of these vectors. Additionally, we prove an upper bound for the algebraic eigenvalues of a tropical matrix, using the tropical Rayleigh quotients.\end{abstract}



\setcounter{tocdepth}{3}

\section{Introduction}
\subsection{motivation}
Tropical semiring is the set $\rmax:= \R\cup\{-\infty\}$ equipped with addition $a \oplus b = \max\{a,b\}$ and multiplication $a \otimes b = a+b$, with the neutral element (of addition) $\varepsilon:=-\infty$. This algebra is known by various names, including max-plus algebra or max algebra. In general, tropical algebra deals with algebraic systems where max or min is used instead of addition. For more information about tropical algebra and its applications in discrete event systems, combinatorial optimization and scheduling, mathematical modeling and also algebraic geometry, see~\cite{baccelli1992synchronization,butkovivc2010max, maclagan2015introduction,heidergott2006max,akian2006max}.

For any square tropical matrix, there exist two different types
of tropical eigenvalues: geometric and algebraic. \new{Tropical geometric eigenvalues} are the values which satisfy the standard eigenvalue-eigenvector equation. The number
of tropical geometric eigenvalues of a matrix cannot exceed the number of strongly connected components of its associated digraph \cite{butkovivc2010max}. On the other hand, \new{tropical algebraic eigenvalues} are the \new{tropical roots} of the tropical characteristic polynomial of the matrix. Every $n$-by-$n$ tropical matrix has $n$ tropical algebraic eigenvalues, counting multiplicities \cite{akian2006max}. Due to the absence of additive inverse in tropical semirings, these two notions are not equivalent, although there are some connections between them. In particular, every tropical geometric eigenvalue is also a tropical algebraic eigenvalue, and the largest tropical geometric eigenvalue coincides with the largest tropical algebraic eigenvalue \cite{cuninghame1983characteristic}.

Tropical algebraic eigenvalues are important because various results relate the classical eigenvalues of the matrix to the tropical algebraic eigenvalues. Actually, tropical algebraic eigenvalues give  order-of-magnitude approximations to the moduli of the classical eigenvalues of matrices. Such results have recently received more attention because of modern developments in tropical algebra and numerical analysis \cite{akian402090min, akian2014tropical, tavakolipour2020asymptotics,marchesini2015tropical,  tavakolipour2018tropical}. Although not directly related to the present work, it is worth noting that tropical roots are of considerable numerical interest due to their computational efficiency and their widespread use in approximating the moduli of eigenvalues of matrix polynomials~\cite{akian2004perturbation,noferini2015tropical,bini2013locating, akian2016non, sharify2011scaling}. 

 As we mentioned before, unlike in the classical linear algebra, for $A \in \rmax^{n \times n}$ and the tropical algebraic eigenvalue $\lambda$ there is not necessarily a vector  $x \neq (\varepsilon, \ldots, \varepsilon)^T$ which satisfies the classical eigenvalue-eigenvector equation. This problem is a difficulty in eigenproblem of tropical linear algebra. Some authors have investigated this problem in order to find an alternative relation for the standard eigenvalue-eigenvector relation in this algebra.
 
  Some notions of generalized eigenvectors associated with eigenvalues over $\rmax$ have already been investigated in the literature. In particular, Nishida and co-authors have used the coefficient
of the tropical characteristic polynomial $P_A$ of the square matrix $A$. They have used the fact that each coefficient
of the characteristic polynomial of $A$ corresponds to the weight of
a multi-circuit in the associated graph. They have called a multi-circuit $\lambda$-maximal if the corresponding term attains the maximum of the $P_A(\lambda)$. The basic condition for a matrix to be in this way is that for each $\lambda \in \rmax$ the
lengths of any two distinct $\lambda$-maximal multi-circuits be different which makes some challenges for symmetric matrices ~\cite{Nishida2020,nishida2021independence}. 

Some other authors have used some extensions of tropical semiring in order to find some notions of eigenvectors. Particularly, supertropical algebra in the work of Izhakian and Rowen~\cite{izhakianmatrix3} and symmetrized tropical algebra in the work of Akian and others~\cite{akian2025spectral}.

\subsection{Generalized tropical eigenvector}
In \cite{tavakolipour2019numerical}, the authors defined \new{tropical numerical range} similarly to the ordinary numerical range (\Cref{nr_def}). They showed that tropical numerical range is a set which contains all the tropical algebraic eigenvalues. Moreover,  they proved for the $n$ by $n$ matrix $A = (a_{ij})$ over $\rmax$, the tropical numerical range is equal to the closed interval $[\min_i a_{ii}, \max_{i,j} a_{ij}]$.

 In this article using the tropical numerical range, we define a generalization of the tropical eigenvector corresponding to the tropical algebraic eigenvalues. For the algebraic tropical eigenvalue $\lambda \in \rmax$,  there exist the \new{generalized tropical eigenvector} $x$ of size $n \times 1$ and $x \neq (\varepsilon, \ldots, \varepsilon)^T$ which satisfy the generalized eigenvalue-eigenvector equation 
 \[x^{T} \otimes  A\otimes    x = \lambda\otimes x^T\otimes x,\]
 We prove that for each algebraic eigenvalue of $A$, there exists at least one generalized tropical eigenvector $x \neq (\varepsilon, \ldots, \varepsilon)^T$ over  $\rmax$. A low-complexity algorithm is proposed for the construction of these vectors.
Finally, we prove a theorem analogous to the classical Rayleigh quotient, utilizing the properties of tropical algebraic eigenvalues and generalized eigenvectors (\Cref{first_inequality}). In contrast to the classical Rayleigh theorem, the tropical version is valid even if the matrix is not symmetric.

The remainder of this paper is organized as follows. Section 2 reviews the necessary preliminaries, covering tropical algebra, characteristic polynomials and tropical eigenvalues. Section 3 is devoted to the definition and properties of the tropical numerical range. In Section 4, we introduce and analyze the properties of the generalized tropical eigenvector. Section 5 presents our efficient algorithm for the computation of these generalized eigenvectors. Section 6 establishes the tropical Rayleigh quotient upper bound for tropical algebraic eigenvalues. Finally, Section 7 concludes the paper, summarizing our contributions and outlining future work.

\section{preliminaries}

In this section, we review fundamental definitions and results concerning tropical algebra, specifically focusing on tropical polynomials and their roots. We further discuss the critical distinction between tropical geometric and algebraic eigenvalues. For more information see \cite{baccelli1992synchronization,butkovivc2010max,akian2024}.

\subsection{Tropical Algebra, polynomials, and roots}

The Tropical Semiring, denoted $\rmax$, is the set $\mathbb{R} \cup \{-\infty\}$ equipped with two operations: tropical addition ($a \oplus b = \max\{a, b\}$) and tropical multiplication ($a \otimes b = a + b$).     The identity element for tropical addition (the zero element, $\varepsilon$) is $-\infty$.
The following identities are fundamental in this paper: the neutral element ($\varepsilon$) is absorbing for $\otimes$ ($a \otimes \varepsilon = \varepsilon$); the tropical power $a^{\otimes b}$ corresponds to standard multiplication ($a \times b$) for $a, b \neq \varepsilon$; and the tropical inverse of any non-$\varepsilon$ element $a$ is its negation ($a^{\otimes -1} = -a$).

A \textit{formal polynomial} $P$ over $\rmax$ (a max-polynomial) is defined like in the classical algebra. To any formal polynomial $P$ of degree $n$, we associate a \textit{polynomial function}. 
\begin{definition}[Tropical Roots]
The non-$\varepsilon$ \new{tropical roots} of $P$ are defined as the points at which the maximum of the associated polynomial function is attained at least twice (i.e., by at least two distinct monomials). If $P$ lacks a constant term, then $P$ is said to possess a tropical root at $\varepsilon$.
\end{definition}
Equivalently, the non-$\varepsilon$ tropical roots of $P$ correspond precisely to the points of non-differentiability of the associated polynomial function when restricted to $\R$. A well-known result states that every formal polynomial $P$ over $\rmax$ of degree $n$ has exactly $n$ tropical roots \cite{cuninghame1980algebra}.

\subsection{Tropical Matrix Operations and Characteristic Polynomial}

Extending the tropical operations $\oplus$ and $\otimes$ to vectors and matrices enables the definition of classical linear algebra constructs: matrix--matrix, matrix--vector, and matrix--scalar products, as well as matrix addition, all analogous to classical linear algebra. We denote the set of $m \times n$ matrices over $\rmax$ by $\rmax^{m \times n}$ and the set of $n \times 1$ vectors by $\rmax^n$. Also, for $n \in \N$, $[n]$ denotes the set $\{1, \ldots,n\}$.

\begin{definition}[Tropical Algebraic Eigenvalue] \label{algebraic}
Let $A \in \rmax^{ n \times n}$. The \new{tropical algebraic eigenvalues} of $A$, denoted $\lambda_{1}(A)\leq \cdots\leq \lambda_{n}(A)$, are the tropical roots of its tropical characteristic polynomial $P_A$.
\end{definition}

The term \emph{algebraic} is employed here because a tropical algebraic eigenvalue $\lambda$ is not guaranteed to satisfy the eigenvalue-eigenvector equation $A\otimes x = \lambda\otimes x$ for some non-zero vector $x \in \rmax^{n}$ (i.e., $x \neq (\varepsilon, \ldots, \varepsilon)^T$).

\begin{remark}
Every tropical geometric eigenvalue is a tropical algebraic eigenvalue, but the converse is not always true.
 \end{remark}
 \begin{remark}
 The greatest tropical algebraic eigenvalue is equal to the greatest tropical geometric eigenvalue.
  \end{remark}
\section{Tropical numerical range}
In this section, we recall several essential results from \cite{tavakolipour2019numerical}. Throughout the following discussion, we use the max‑norm of a vector 
$v=(v_i)\in \rmax^n$, with $v  \neq (\varepsilon, \ldots, \varepsilon)^T$, as introduced in \cite[p.60]{butkovivc2010max}, defined by
\[\Vert v \Vert = \max \{v_1, \ldots, v_n\}\;.\]
A vector $v$ is said to be \new{scaled} whenever $\Vert v\Vert=0$. Assume that $\max\{v_1, \ldots, v_n\} = v^*$. Then the vector 
\begin{equation}\label{scaled}(v^*)^{\otimes-1} \otimes v =(v_1 - v^*, \ldots, v_n - v^*)\;,\end{equation}
 represents the scaled form of $v$.

\begin{definition}[Tropical numerical range]\label{nr_def}
Suppose that $A \in \mathbb{R}_{\max}^{n \times n}$. The tropical numerical range (field of values) of $A$ is defined as
\[
F_{\max}(A):=\{ (x^{T} \otimes A \otimes x) \otimes(x^{T} \otimes x)^{\otimes -1}: x \in \mathbb{R}_{\max}^{n}, x  \neq (\varepsilon, \ldots, \varepsilon)^T\}.
\]
\end{definition}
\noindent Note that 
\begin{equation}\label{expand}
x^{T} \otimes A \otimes x=\max\limits_{i,j \in [n]} x_{i}+a_{ij}+x_{j}.
\end{equation}
and by letting $x^*=\max\{x_1, \ldots, x_n\}$
\begin{equation}\label{expand2}
x^{T} \otimes x=2x^*.
\end{equation}
\begin{lemma}\label{normal} Let \( A \in \mathbb{R}_{\max}^{n \times n} \). Then the set of values \( \{x^T \otimes A \otimes x : x \in \mathbb{R}_{\max}^n, \Vert x \Vert = 0\} \) is equal to \( F_{\max}(A) \). \end{lemma}

\begin{theorem}\label{interval}
Let $A= (a_{ij})\in \mathbb{R}_{\max}^{n\times n}$. We have
\[F_{\max}(A)=\bigg[\min \limits_{i \in [n] }a_{ii},\max \limits_{i,j \in  [n]}a_{ij}\bigg].\]
\end{theorem}
\begin{theorem}\label{contain_eigs} If \( A \in \mathbb{R}_{\max}^{n \times n} \), then each tropical algebraic eigenvalue of \( A \) is an element of \( F_{\max}(A) \). \end{theorem}

\section{Generalized tropical eigenvector}
In this section, following the statement of \Cref{exist}, we define the \new{generalized tropical eigenvector}. Subsequently, in \Cref{leq,the_second}, we derive explicit, computationally cheap formulas for determining this generalized tropical eigenvector. These derived formulas are conveniently summarized in \Cref{fig_formula}.

\begin{theorem}\label{exist}
Let $A \in \mathbb{R}_{\max}^{n \times n}$ and $\lambda \in \rmax$ be a tropical algebraic eigenvalue of $A$. Then there exists $x=(x_i) \in \rmax^n$ such that $x  \neq (\varepsilon, \ldots, \varepsilon)^T$ and the following equation holds:
\[x^{T} \otimes  A\otimes    x = \lambda\otimes x^T\otimes x\;.\]
Furthermore, let $y$ denote the scaled form of $x$. Then we have  
\[ y^{T} \otimes A   \otimes y = \lambda\;.\]
\end{theorem}
\begin{proof}
The existence of $x \in \mathbb{R}_{\max}^n$ such that $x  \neq (\varepsilon, \ldots, \varepsilon)^T$ is guaranteed by Theorem \ref{contain_eigs}. Furthermore, by defining $x^* = \max\{x_1, \ldots, x_n\}$ and applying results from \Cref{expand,expand2}, we have the following equation
\[\max_{i,j\in [n]}x_{i}+a_{ij}+x_{j}=\lambda+2x^*\;.\]
Rearranging the terms yields:
\[\max_{i,j\in [n]} x_{i}+a_{ij}+x_{j}-2x^*=\lambda\;.\]
Consequently, by defining the scaled vector $y = (y_i)$ according to \Cref{scaled} such that $y_i = x_i - x^*$, the equation transforms into
\[\max_{i,j\in [n]} (x_{i}-x^*)+a_{ij}+(x_{j}-x^*)=\max_{i,j\in [n]}  y_i+a_{ij}+y_j=\lambda\;,\]
which successfully yields the desired result.
\end{proof}
\begin{definition}[Generalized tropical eigenvector]
Let $A \in \mathbb{R}_{\max}^{n \times n}$ and $\lambda \in \rmax$ be a tropical algebraic eigenvalue of $A$. Then the vector $x \in \rmax^n, x  \neq (\varepsilon, \ldots, \varepsilon)^T$
 satisfying
 \[x^{T} \otimes  A\otimes    x = \lambda\otimes x^T\otimes x\;,\]
is called a generalized tropical eigenvector associated with the eigenvalue $\lambda$.
\end{definition}

\section{Computation of generalized tropical eigenvectors}

In what follows, we derive an explicit formula for the generalized tropical eigenvector corresponding to the tropical eigenvalues of $A \in \mathbb{R}_{\max}^{n \times n}$. To ensure clarity and ease of understanding, we summarize everything briefly in \Cref{fig_formula}. Let $A=(a_{ij})\in \rmax^{n\times n}$. In \Cref{leq}, we investigate the conditions under which all diagonal entries of $A$ are less than or equal to its tropical algebraic eigenvalue $\lambda$. The theorem also requires the existence of an entry $a_{pq}$ such that $a_{pq}\geq \lambda$. We note that this condition is always satisfied, since by \Cref{interval}, the maximum tropical algebraic eigenvalue of 
$A$ does not exceed the maximum entry of $A$. \begin{figure}[!h]
\begin{center}
\scriptsize
\begin{forest}
  for tree={
    align=center,
    parent anchor=south,
    child anchor=north,
    font=\sffamily,
    edge={thick, -{Stealth[]}},
    l sep+=10pt,
    edge path={
      \noexpand\path [draw, \forestoption{edge}] (!u.parent anchor) -- +(0,-10pt) -| (.child anchor)\forestoption{edge label};
    },
    if level=0{
      inner xsep=0pt,
      tikz={\draw [thick] (.south east) -- (.south west);}
    }{}
  }
  [ $A\mbox{\;=:}( a_{ij})\in \rmax^{n\times n}\mbox{,\;}\; \lambda \in \rmax\mbox{,\;} \;x \mbox{=\;} (x_{i}) \in \rmax^n$
    [\Cref{leq}:\; $(\forall i\in{[n]} \mbox{,\;} a_{ii}\leq \lambda )$
        [$a_{pq}\geq \lambda \mbox{,\;} a_{pq}\geq a_{qp}$
            [$x_p\mbox{=\;}0 \mbox{,\;} x_q\mbox{=\;}\lambda- a_{pq}$]
        ]
        [$a_{pq}\geq \lambda \mbox{,\;}a_{pq}\leq a_{qp}$
            [$x_q\mbox{=\;}0 \mbox{,\;} x_p\mbox{=\;}\lambda- a_{qp}$]
        ]
    ] 
    [\Cref{the_second} : $(a_{qq}\geq \lambda\mbox{,\;}\min_{i\in [n]}a_{ii}\mbox{=\;} a_{pp})$
        [ $\lambda+a_{qq}\leq 2\max\{a_{pq}\mbox{,}a_{qp}\}$
            [ $x_p\mbox{=\;}0\mbox{,\;}x_q \mbox{=\;} \lambda - \max\{a_{pq}\mbox{,}a_{qp}\}$]
        ] 
        [$\lambda+a_{qq}\geq 2\max\{a_{pq}\mbox{,}a_{qp}\}$
            [ $x_p\mbox{=\;}0\mbox{,\;}x_q\mbox{=\;} \frac{\lambda - a_{qq}}{2}$]
        ]
    ] 
  ] 
\end{forest}
\end{center}
\caption{Explicit formula for the generalized tropical eigenvector corresponding to the eigenvalue $\lambda$.}\label{fig_formula}
\end{figure}
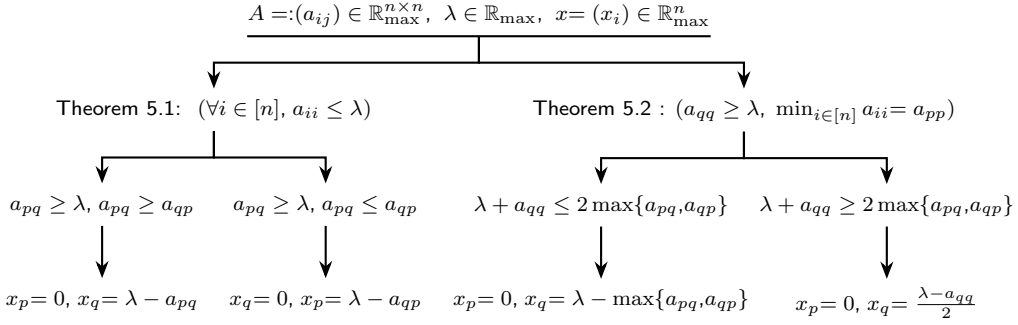

\begin{theorem}\label{leq}
Let $A=(a_{ij})\in \rmax^{n\times n}$ be a matrix with a tropical eigenvalue $\lambda \in \rmax$, such that $a_{ii}\leq \lambda$ for all $i\in[n]$.
\begin{enumerate}
    \item If $p,q\in[n]$ satisfy $a_{pq}\geq \lambda$ and $a_{pq}\geq a_{qp}$, define $x=(x_i)\in \rmax^n$ by $x_p=0$, $x_q=\lambda- a_{pq}$, and $x_{i}=\varepsilon$ for $i \in[n]\setminus{p,q}$.
    \item If $p,q\in[n]$ satisfy $a_{pq}\geq \lambda$ and $a_{qp}\geq a_{pq}$, define $x=(x_i)\in \rmax^n$ by $x_q=0$, $x_p=\lambda- a_{qp}$, and $x_{i}=\varepsilon$ for $i \in[n]\setminus{p,q}$.
\end{enumerate}
In both cases, $x$ is a generalized tropical eigenvector in $\rmax^n$ associated with $\lambda$, with $\Vert x\Vert=0$.
\end{theorem}
\begin{proof}
The proof for the first case is presented; the second case follows by analogous arguments.
By construction, $x_p=0$ and $x_q\le 0$ (since $x_q=\lambda-a_{pq}\le 0$), and all other entries of $x$ are equal to $\varepsilon$. Hence, the maximum entry of $x$ is $0$, so that  $\Vert x\Vert=0$.
On the other hand,  by \Cref{expand} for $p\neq q$ we have 
\begin{eqnarray}x^T \otimes A \otimes x &=& \max_{i,j \in [n]}\{x_i +a_{ij}+x_j\} \nonumber\\
&=& \max \{x_p+x_q+a_{pq},x_q+x_p+a_{qp}, 2x_p+a_{pp},2 x_q+a_{qq}, \varepsilon\}\nonumber\\
&=&\max\{\lambda , \lambda-a_{pq}+a_{qp}, a_{pp}, 2\lambda-2a_{pq}+a_{qq}\}= \lambda, \nonumber\end{eqnarray}
where the last equality follows from the fact that
\[ \lambda-a_{pq} +a_{qp}  \leq \lambda,\; a_{pp}  \leq \lambda, \;2\lambda-2a_{pq}+a_{qq}  \leq \lambda.\]
Now, assume $p=q$. Then $a_{pq} = a_{pp} = a_{qq} \ge \lambda$ and also $a_{pq} = a_{pp} = a_{qq} \le \lambda$, so 
\begin{equation}\label{equal}a_{pq} = a_{pp} = a_{qq}=\lambda.\end{equation}
Hence $x_p = x_q = 0$ and $x_i = \varepsilon$ for $i \neq p = q$. Therefore,
\begin{eqnarray}x^T \otimes A \otimes x &=& \max_{i,j \in [n]}\{x_i +a_{ij}+x_j\} \nonumber\\
&=& \max \{x_p+x_q+a_{pq},x_q+x_p+a_{qp}, 2x_p+a_{pp},2 x_q+a_{qq}, \varepsilon\}\nonumber\\
&=&\max\{a_{pq}, a_{qp},a_{pp},a_{qq}, \varepsilon\}=\lambda\;,\nonumber\end{eqnarray}
where the last equality is by \Cref{equal}. This establishes the result.
\end{proof}
In the following theorem, we consider the case in which at least one diagonal entry of 
$A$ is greater than  or equal to $\lambda$.

\begin{theorem}\label{the_second}
Let $A=(a_{ij})\in \rmax^{n\times n}$ have tropical algebraic eigenvalue $\lambda \in \rmax$. Assume that there exists $q\in[n]$ such that $a_{qq}\geq \lambda$, and let $p\in[n]$ satisfy
\[\min_{i\in [n]}a_{ii}=a_{pp}.\]
Define $x_p=0$ and $x_i=\varepsilon$ for all $i\in[n]\setminus\{p,q\}$. To determine $x_q$, consider the following two cases:
\begin{enumerate}
\item If $\lambda+a_{qq}\le 2\max\{a_{pq},a_{qp}\}$, set
\[x_q = \lambda - \max\{a_{pq}, a_{qp}\}.\]
\item If $\lambda+a_{qq}\geq2\max\{a_{pq},a_{qp}\}$, set
\[x_q = \frac{\lambda - a_{qq}}{2}.\]
\end{enumerate}
Then, the vector $x=(x_i)\in\rmax^n$ is a generalized tropical eigenvector associated with $\lambda$, such that $\Vert x \Vert =0$.
\end{theorem}

\begin{proof}
\textbf{Proof of the first part: }First let $p\neq q$. Since $\lambda +a_{qq}\leq 2\max\{a_{pq}, a_{qp}\}$ and $ \lambda \leq a_{qq}$, so 
\begin{equation}\label{cond1}\lambda \leq \max\{a_{pq}, a_{qp}\}.
\end{equation} 
 Therefore, by the definition of $x$, we have $x_p=0$, and by \Cref{cond1}
\[
x_q=\lambda-\max\{a_{pq},a_{qp}\}\le 0 \quad,
\]
and all other components are equal to $\varepsilon$. Hence, the maximum entry of $x$ is $0$, and thus $\|x\|=0$.
Moreover,
\begin{eqnarray}
x^T \otimes A \otimes x 
&=& \max_{i,j \in [n]}\{x_i + a_{ij} + x_j\} \nonumber\\
&=& \max \big\{ x_p + x_q + a_{pq},\; x_q + x_p + a_{qp},\; 2x_p + a_{pp},\; 2x_q + a_{qq},\; \varepsilon \big\} \nonumber\\
&=& \max \Big\{ \lambda - \max\{a_{pq}, a_{qp}\} + a_{pq},\;
\lambda - \max\{a_{pq}, a_{qp}\} + a_{qp},\;
a_{pp},\;
2\lambda - 2\max\{a_{pq}, a_{qp}\} + a_{qq} \Big\} \nonumber\\
&=& \lambda\quad, \nonumber
\end{eqnarray}
where the last equality follows from the facts that
\[
\lambda - \max\{a_{pq}, a_{qp}\} + a_{pq} \le \lambda, \qquad
\lambda - \max\{a_{pq}, a_{qp}\} + a_{qp} \le \lambda\quad,
\]
and at least one of these two expressions is equal to $\lambda$. Moreover, by \Cref{interval}, we have 
\[\min_{i\in [n]}a_{ii} =a_{pp} \le \lambda\quad.\]
Finally, using the assumption in the first part, we obtain
\[
2\lambda - 2\max\{a_{pq}, a_{qp}\} + a_{qq}
= \lambda + \big(\lambda + a_{qq} - 2\max\{a_{pq}, a_{qp}\}\big)
\le \lambda.
\]
 Now assume that $p=q$.  Then $a_{pp}=a_{qq}=a_{pq}=a_{qp} $. Since $a_{pp} \leq \lambda$ and $a_{qq}\geq \lambda$ we have 
\[a_{pp}=a_{qq}=a_{pq}=a_{qp}=\lambda\quad,\]
and $x_p=x_q=0$.
So $\Vert x \Vert =0$ and
 \begin{eqnarray}
x^T \otimes A \otimes x 
&=& \max_{i,j \in [n]}\{x_i + a_{ij} + x_j\} \nonumber\\
&=& \max \big\{ x_p + x_q + a_{pq},\; x_q + x_p + a_{qp},\; 2x_p + a_{pp},\; 2x_q + a_{qq},\; \varepsilon \big\} \nonumber\\
&=&\max\big\{\lambda,\; \varepsilon \big\}=\lambda\quad,\nonumber
\end{eqnarray}
 which completes the proof of the first part.
 \\
\textbf{Proof of the second part: } First let $p\neq q$. Since $\lambda \leq a_{qq}$ we have $x_q=\frac{\lambda - a_{qq}}{2}\leq0$. 
So, from the way $x$ is defined, $x_p = 0$, $x_q \le 0$, and all other entries are $\varepsilon$. Thus, the largest entry of $x$ is $0$, giving $\Vert x\Vert  = 0$. Moreover, 
 \begin{eqnarray}x^T \otimes A \otimes x &=& \max_{i,j \in [n]}\{x_i +a_{ij}+x_j\} \nonumber\\
&=& \max \{x_p+x_q+a_{pq},x_q+x_p+a_{qp}, 2x_p+a_{pp},2 x_q+a_{qq}, \varepsilon\}\nonumber\\
&=&\max\{\frac{\lambda-a_{qq}+2a_{pq} }{2},\frac{\lambda-a_{qq}+2a_{qp} }{2},  a_{pp}, \lambda,\varepsilon\}= \lambda, \nonumber\end{eqnarray}
where the last equality follows from the facts that
$a_{pp}\leq \lambda$ and also
since 
\[\lambda+a_{qq} \geq 2\max\{a_{pq}, a_{qp}\}\quad,\]  
therefore 
\begin{equation}\label{cond3}
2\max\{a_{pq}, a_{qp}\}  -a_{qq} \leq \lambda\quad.
\end{equation}
So by \Cref{cond3}, 
\[\max\{\frac{\lambda-a_{qq}+2a_{pq} }{2},\frac{\lambda-a_{qq}+2a_{qp} }{2}\}\leq \frac{\lambda + (2\max\{a_{pq}, a_{qp}\}  -a_{qq})}{2}\leq \frac{2\lambda}{2}=\lambda\quad.\]
So we have the result. The proof fore the case $p=q$ is similar to the first part.

\end{proof}

\begin{corollary}
Let $A \in \rmax^{n \times n}$. By \Cref{leq,the_second}, for every tropical algebraic eigenvalue $\lambda \in \rmax$, one can explicitly construct a generalized tropical eigenvector associated with $\lambda$.
\end{corollary}
\begin{example}
Consider the matrix \[A = \left(\begin{array}{cccc}
3&1&3&5\\
2&0&0&0\\
-3&0&0&2\\
4&1&3&5
\end{array}\right)\;.\]
The algebraic tropical eigenvalues of $A$ are 
\[\lambda_1=0,\; \lambda_2=2,\; \lambda_3= 4,\; \lambda_4=5\;.\] 
\begin{enumerate}
\item (Eigenvector corresponding to $\lambda_1=0$) To find a generalized tropical eigenvector associated with $\lambda_1$, we apply \Cref{the_second}. We have the $\min_{i \in [4]} a_{ii}=a_{33}=0$. Choose $p=3$ and $q=4$. Then
\[\lambda_1+a_{qq}=0+5\leq 2  \max\{a_{34},a_{43}\}=2\times 3= 6\;, \]
By the first part of \Cref{the_second}, we obtain
  \[x_q = 0-\max\{a_{34},a_{43}\}=0-3 = -3\;.\] 
Hence
 $x= \left(\begin{array}{cccc}
\varepsilon&
\varepsilon&
0&
-3
\end{array}\right)^T$.
\item (Eigenvector corresponding to $\lambda_2=2$) Again applying \Cref{the_second} with $p=3$ and $q=4$:
 \[\lambda_2+a_{qq}=2+5\geq 2  \max\{a_{34},a_{43}\}=2\times 3= 6\;. \] 
Thus by the second part of \Cref{the_second}
\[x_q = \frac{2-5}{2}=-1.5.\;\]
 So, $x = \left(\begin{array}{cccc}
\varepsilon&
\varepsilon&
0&
-1.5
\end{array}\right)^T$.
\item (Eigenvector corresponding to $\lambda_3=4$) Following the same steps with $p=3$ and $q=4$ and using \Cref{the_second} we have 
\[\lambda_3+a_{qq}=4+5\geq 2  \max\{a_{34},a_{43}\}=2\times 3= 6\;.\] 
By second part of \Cref{the_second} 
\[x_q = \frac{4-5}{2}=-0.5.\;\]
Therefore $x = \left(\begin{array}{cccc}
\varepsilon&
\varepsilon&
0&
-0.5
\end{array}\right)^T$.
\item (Eigenvector corresponding to $\lambda_4= 5$) Finally, we use \Cref{leq} with $a_{pq}=a_{14}=5$. According to the first part of \Cref{leq} we have  
\[x_p=x_1 = 0, \; x_q=x_4= 5-5= 0.\;\]
Hence $x = \left(\begin{array}{cccc}
0&
\varepsilon&
\varepsilon&
0
\end{array}\right)$.
\end{enumerate}
\end{example}
\section{Tropical Rayleigh quotient upper bound for $\lambda_k$ }
The tropical Rayleigh quotient approach reveals an important property of the tropical
 algebraic eigenvalues. While the classical proof for the Rayleigh quotient theorem required the fact that the matrix be Hermitian to guarantee the upper bound, we demonstrate that no such symmetry assumption is needed for tropical matrices. In the following, we first recall the classical Rayleigh quotient  theorem and then two definitions over tropical algebra, after which we will prove an important property concerning the tropical generalized eigenvector.
  \begin{theorem}[Classical Rayleigh quotient  theorem]
Let $A\in \C^{n \times n}$ be Hermitian matrix and let $\lambda_1 \leq \cdots \leq \lambda_n$
be its algebraically ordered eigenvalues. Let $k\in[n]$ and denote by $x_i$ the classical eigenvector corresponding to $\lambda_i$, and let 
$S = \spn\{x_k, \ldots, x_n\}$. Then
\[\lambda_k=\min_{\substack{ u\in S\\ u\neq (0, \ldots, 0)^T}} \frac{u^TAu}{u^Tu}\]
\end{theorem}
\begin{definition}[\cite{butkovivc2010max}]
A vector $v=(v_1, \cdots,v_n) ^T\in \rmax^n$ is called a max-combination of $S \subseteq \rmax^n$ if $v =\sum_{x \in S}^{\oplus}\alpha_x\otimes x, \; \alpha_x \in \rmax$, where only a finite number of $\alpha_x$ are finite.  
\end{definition}
\begin{definition}\cite{butkovivc2010max}
The set of all max-combinations of $S$ is denoted by $\spn(S)$.
\end{definition}

Extending the classical Rayleigh quotient theorem to the tropical setting, we derive a new result using tropical algebraic eigenvalues and generalized eigenvectors. This theorem is significant because it holds universally for non-symmetric matrices. \begin{theorem}\label{first_inequality}
Let $A \in \rmax^{n \times n}$ be a matrix whose tropical algebraic eigenvalues satisfy 
$\lambda_1 \leq \cdots \leq \lambda_n$. 
Denote by $x_i$ the scaled generalized tropical eigenvector corresponding to $\lambda_i$, and let 
$S = \spn\{x_k, \ldots, x_n\}$. 
Then the following equality holds:
\[
\min_{\substack{u \in S \\ u  \neq (\varepsilon, \ldots, \varepsilon)^T}} 
\left( (u^T \otimes A \otimes u) \otimes (u^T \otimes u)^{\otimes -1} \right) 
= \lambda_k.
\]
\end{theorem}
\begin{proof}
Throughout the proof we use $k\!:\!n$ to denote the index set $\{k, \ldots, n\}$. Let $u \in S$. Then $u$ is the max-combination of $S$. So $u = \sum^{\oplus}_{i\in k:n}c_ix_i$. Hence
\begin{eqnarray}\label{first}
u^T \otimes A \otimes u &=& \bigg( \sum^{\oplus}_{i\in k : n}c_i \otimes x_i ^T\bigg) \otimes A  \otimes \bigg( \sum^{\oplus}_{j\in k: n}c_j  \otimes x_j\bigg)\nonumber\\
&=&\bigg(\sum_{i= j, \; i,j \in k : n}^{\oplus}c_i\otimes x_i^T \otimes A \otimes c_j \otimes x_j\bigg)\oplus \bigg(\sum_{i\neq j, \; i,j\in k : n}^{\oplus}c_i\otimes x_i^T\otimes A \otimes c_j \otimes x_j\bigg)\nonumber\\
&=&\bigg(\sum_{i \in k : n}^{\oplus}c_i \otimes c_i  \otimes x_i^T\otimes  A \otimes x_i\bigg)\oplus (\sum_{i\neq j, \; i,j\in k : n}^{\oplus}c_i \otimes c_j \otimes x_i^T\otimes  A  \otimes x_j\bigg)\nonumber\\
&=& \sum_{i\in k : n}c_i^{\otimes 2}\lambda_i \oplus  \bigg(\sum_{i\neq j, \; i,j \in k : n}^{\oplus}c_i \otimes x_i^T  \otimes A  \otimes c_j  \otimes x_j\bigg)\nonumber\\
&\geq & \bigg(\lambda_k\otimes \sum_{i=k:n}c_i^{\otimes 2} \bigg)\oplus\bigg( \sum_{i\neq j, i,j \in k : n}c_i \otimes x_i^T \otimes A \otimes x_j \otimes c_j\bigg)\nonumber\\
&\geq &\lambda_k \otimes  \sum_{i=k:n}c_i^{\otimes 2}.\nonumber\end{eqnarray}
where the last inequality follows directly from the maximization property of $\oplus$. Therefore we have 
\begin{equation}
u^T \otimes A \otimes u \geq \lambda_k \otimes  \sum_{i=k:n}c_i^{\otimes 2}\;.
\end{equation}
Also we have
\begin{eqnarray}
u^T\otimes u&=&  \big( \sum^{\oplus}_{i\in k : n}c_i \otimes x_i ^T\big)\otimes\big( \sum^{\oplus}_{j\in k : n}c_j\otimes x_j\big)\nonumber\\
&=&\sum^{\oplus}_{i\in k : n}c_i^{\otimes 2} \oplus \sum_{i\neq j, i,j \in k : n}c_i  \otimes x_i^T\otimes c_j \otimes x_j\;,\nonumber
\end{eqnarray}
where the last equality is by the fact that $\Vert x_i \Vert =0$ and so $x_i^T\otimes x_i =0$.
Since $\Vert x_i \Vert = \Vert x_j \Vert =0$, therefore $q_{ij}:=x_i^T\otimes x_j \leq 0$. So,
\[\sum_{i\neq j, i,j \in k : n}c_i\otimes x_i^T\otimes  c_j \otimes x_j=\max_{i\neq j, i,j \in k : n}\{c_i+c_j+q_{ij}\} \leq \max_{i\in k : n}\{2c_i\}=2\max_{i\in k : n}\{c_i\}=\sum^{\oplus}_{i\in k : n}c_i^{\otimes 2}\;.\]
Thus 
\begin{equation}\label{second}
u^{T}\otimes u = \sum^{\oplus}_{i\in k : n}c_i^{\otimes 2}.
\end{equation}

Combining \Cref{first_inequality,second} yields
\begin{equation}\label{geq}\lambda_k \leq \min_{u\in S, u \neq (\varepsilon, \ldots, \varepsilon)^T} (u^T \otimes A\otimes u)\otimes (u^T \otimes u)^{\otimes -1}\;.\end{equation}
Now, setting $u =x_k$  gives
 \begin{equation}\label{min}(x_k^T \otimes A\otimes x_k)\otimes (x_k^T \otimes x_k)^{\otimes -1}=\lambda_k.\end{equation}
Finally, combining \Cref{geq,min} gives the  result.
\end{proof}





\end{document}